\documentclass[11pt]{amsart}
\usepackage{amsfonts}
\usepackage{amssymb}
\usepackage{graphicx,color}

\usepackage{a4wide}

\usepackage{amsmath}

\usepackage{amsthm}

\graphicspath{{./Figs/}}
\DeclareGraphicsExtensions{.pdf,.jpg,.png,.eps,.gif}
\title[Energy of spherical designs]{\bf Universal upper and lower bounds on energy of spherical designs}
\def\ds{\displaystyle}
\date{\today}

\newtheorem{theorem}{Theorem}[section]
\newtheorem{lemma}[theorem]{Lemma}
\newtheorem{corollary}[theorem]{Corollary}

\theoremstyle{definition}
\newtheorem{defn}[theorem]{Definition}
\newtheorem{example}[theorem]{Example}
\newtheorem{remark}[theorem]{Remark}

\newcommand{\Sp}{\mathbb{S}}

\author[P. Boyvalenkov]{P. G. Boyvalenkov $^\dagger$}
\address{Institute of Mathematics and Informatics, Bulgarian Academy of Sciences,
8 G Bonchev Str.,
1113  Sofia, Bulgaria \\
and Faculty of Mathematics and Natural Sciences, South-Western University, Blagoevgrad, Bulgaria.
}
\email{peter@math.bas.bg}
\thanks{\noindent $^\dagger$ The research of this author was supported, in part, by a Bulgarian NSF contract I01/0003.}

\author[P. Dragnev]{P. D. Dragnev $^{\dagger \dagger}$}
\address{Department of Mathematical Sciences,
Indiana-Purdue University
Fort Wayne, IN 46805, USA }
\email{dragnevp@ipfw.edu}
\thanks{\noindent $^{\dagger \dagger}$ The research of this author was supported, in part, by a Simons Foundation grant no. 282207.}

\author[D. Hardin]{D. P. Hardin$^*$}
\address{Center for Constructive Approximation, Department of Mathematics, \hspace*{.1in}
Vanderbilt University,
Nashville, TN 37240, USA  }
\email{doug.hardin@vanderbilt.edu}

\author[E. Saff]{E. B. Saff$^*$}
\email{edward.b.saff@vanderbilt.edu}
\thanks{\noindent $^*$ The research of these authors was supported, in part,
by the U. S. National Science Foundation under grants    DMS-1412428 and DMS-1516400.
}
\author[M. Stoyanova]{M. M. Stoyanova$^{**}$}
\address{Faculty of Mathematics and Informatics,
Sofia University,
5 James Bourchier Blvd.,
1164 Sofia, Bulgaria}
\email{stoyanova@fmi.uni-sofia.bg}
\thanks{
\noindent $^{**}$ The research of this author was supported, in part, by the Science Foundation of Sofia University under contract 015/2014 and 144/2015.
}
\thanks{The authors express their gratitude to  the Institute of Mathematics and Informatics at the Bulgarian Academy of Sciences and the Erwin Schr\"{o}dinger International Institute for providing  stimulating environments for collaborative research.}

\begin{document}

\begin{abstract}
Linear programming (polynomial) techniques are used to obtain lower and upper bounds
for the potential energy of spherical designs. This approach gives unified bounds
that are valid for a large class of potential functions. Our lower bounds are optimal for absolutely monotone potentials in the sense
that for the linear programming technique they cannot be improved by using polynomials of the same or lower degree. When
additional information about the structure (upper and lower bounds for the inner products) of the designs is known, improvements on the bounds are obtained. Furthermore, we provide `test functions'  for determining when the linear programming lower bounds for energy can be improved  utilizing higher degree polynomials. We also provide some asymptotic results for these energy bounds.
\end{abstract}


\maketitle
\section{Introduction}


Let $\mathbb{S}^{n-1}$ be the unit sphere in $\mathbb{R}^n$.
We refer to a finite set $C \subset \mathbb{S}^{n-1}$ as a {\em spherical code}
and, for a given (extended real-valued) function $h:[-1,1] \to [0,+\infty]$, we consider
the {\em $h$-energy} (or the potential energy) of $C$ defined by
\begin{equation}
E(n,C;h):=\sum_{x, y \in C, x \neq y} h(\langle x,y \rangle),
\end{equation}
where $\langle x,y \rangle$ denotes the inner product of $x$ and $y$.
At times we shall require $h$ to be {\em absolutely monotone} or {\em strictly absolutely monotone} on $[-1,1)$; i.e.,
its $k$-th derivative satisfies $h^{(k)}(t) \geq 0$ ($h^{(k)}(t)>0$, resp.) for all $k\ge 0$ and $t\in[-1,1)$.

A spherical $\tau$-design $C \subset \mathbb{S}^{n-1}$, whose cardinality we denote by $|C|$,  is a spherical code such that
\begin{eqnarray*}
\frac{1}{\mu(\mathbb{S}^{n-1})} \int_{\mathbb{S}^{n-1}} f(x) d\mu(x)= \frac{1}{|C|} \sum_{x \in C} f(x)
\end{eqnarray*}
($\mu(x)$ is the surface area measure) holds for all polynomials $f(x) = f(x_1,x_2,\ldots,x_n)$ of total degree at most $\tau$.
The maximal number $\tau = \tau(C)$ such that $C$ is a spherical $\tau$-design is called the {\em strength} of $C$.

A commonly arising problem is to estimate the potential energy of certain sets of codes $C.$
In this paper we address this problem
for the class of spherical designs of fixed dimension, strength and cardinality.
Denote by
\begin{equation}\label{LE} \mathcal{L}(n,N,\tau;h):=\inf \{E(n,C;h):|C|=N, \ C \subset \mathbb{S}^{n-1} \mbox{ is a $\tau$-design}\} \end{equation}
and
\begin{equation}\label{UE} \mathcal{U}(n,N,\tau;h):=\sup \{E(n,C;h):|C|=N, \ C \subset \mathbb{S}^{n-1} \mbox{ is a $\tau$-design}\} \end{equation}
the minimum and the maximum possible $h$-energy of a spherical $\tau$-design
of $N$ points on $\mathbb{S}^{n-1}$, respectively. In this paper we derive lower bounds on $\mathcal{L}(n,N,\tau;h)$
and upper bounds on $\mathcal{U}(n,N,\tau;h)$, which then define a strip where the energies of all
designs of fixed dimension, strength and cardinality lie (see Theorem \ref{Thm3.7}).

Concerning lower bounds for energy,
a general linear programming technique  originally introduced by Delsarte \cite{D1} for investigating codes over finite fields (see also \cite{KL}) has  been utilized
 by Yudin \cite{Y}, Kolushov and Yudin \cite{KY} and Andreev \cite{And} (see also \cite{And96,And99,KY94}) to prove the optimality of certain spherical codes among all possible codes.  In 2007, Cohn and Kumar augmented this technique by introducing the notion of \emph{conductivity} to prove the universal optimality of \emph{sharp}\footnote{A code is \emph{sharp} if for some positive integer $m$, the code is a  $2m-1$ design with at most $m$ different values of the distance between distinct points in the code.} codes.  By \emph{universal optimality of a code} we mean that among all codes of the same cardinality, it  minimizes  the energy for all absolutely monotone potentials.  Essential to their method are certain quadrature rules associated with those sharp codes.    Here, by combining the Delsarte method with quadrature formulas developed by Levenshtein, we derive lower bounds for general designs that are optimal for the linear programming
technique when restricted to the use of polynomials of fixed maximal degree.  

 Since the collection of spherical designs is a special subclass of codes for a given dimension and cardinality, we expect that both larger lower bounds and lower upper bounds are possible when compared with minimal energy codes. Indeed we provide such examples in Section 5, where we discuss Mimura 2-designs (see Example~5.4).  Our general upper bounds for $\mathcal{U}(n,N,\tau;h)$ are given in Theorems~\ref{Thm3.6} and \ref{Thm3.7} in Section 3.

We remark that upper bounds for the Riesz $s$-energy of well-separated spherical designs of asymptotically optimal cardinality on $\mathbb{S}^2$ (that is, $|C|=\mathcal{O}(\tau^2)$ as $\tau\to \infty$)  were obtained by Hesse and Leopardi in \cite{H, HL}. The existence of such well-separated  designs of asymptotically optimal cardinality on $\mathbb{S}^{2}$ (and more generally on  $\mathbb{S}^{n-1}$)  has been established recently by Bondarenko, Radchenko, and Vyazovska in \cite{BRV2}.

An outline of our paper is as follows. Some preliminaries are explained in Section 2, where we refer to results and techniques developed by
Delsarte, Goethals and Seidel \cite{DGS} and Levenshtein \cite{Lev2,Lev3,Lev} that will be needed for the statements of our main results.
The relationship between the Delsarte-Goethals-Seidel bounds for the minimum possible size of spherical designs
of prescribed dimension and strength and the Levenshtein bounds for the maximum size of spherical codes
of prescribed dimension and minimum distance will play a very important role in our investigation. Some results on
the structure of designs of fixed dimension, strength and cardinality from \cite{BBD,BBKS} are also discussed.

In Section 3 we formulate two general results that provide the framework for obtaining lower and upper bounds for the energy of spherical designs.
Theorem~\ref{Thm3.1} (lower bounds) is a slight modification of a known result (cf. \cite[Proposition 4.1]{CK},
\cite[Chapter 5]{BHSbook}), but Theorem ~\ref{Thm3.6} (upper bounds) is new.

Theorem~\ref{Thm3.4} gives lower bounds   that are optimal in the sense described in Theorem~\ref{Thm3.5}
-- they cannot be improved by utilizing polynomials of the same or lower degree that satisfy the conditions of Theorem~\ref{Thm3.1}.
Following Levenshtein \cite{Lev} we call these bounds {\em universal}.

Some of the lower bounds from Theorem~\ref{Thm3.4} can be further improved by either restricting the interval containing the inner products of even-strength designs (Theorems~\ref{Thm4.1} and \ref{Thm4.2}) or by allowing polynomials of higher degree for odd-strength designs (Theorems~\ref{Thm4.3} and \ref{Thm4.4} and Corollary~\ref{Cor4.5}). For the latter case,  Theorem~\ref{Thm4.3} provides necessary and sufficient conditions for the global optimality of the bounds from Theorem~\ref{Thm3.4}.

Section 5 is devoted to improving upper bounds for spherical designs utilizing restrictions on their inner products. Some asymptotic results and numerical examples that illustrate these upper bounds are also included.

Finally, in Section 6 we derive an asymptotic lower bound for the energy of spherical designs of fixed strength as the dimension and cardinality grow to infinity in certain relation. An example (Euclidean realization of the Kerdock codes \cite{Ker72}) illustrating the tightness of these asymptotic bounds is presented.

\section{Preliminaries}
The results to be presented in Sections 3-7 utilize the notations and fundamental facts described in the subsections below.
\subsection{Gegenbauer polynomials and the linear programming framework}

For fixed dimension $n$, the Gegenbauer polynomials \cite{Sze} are
defined by  $P_0^{(n)}(t)=1$, $P_1^{(n)}(t)=t$ and the three-term recurrence relation
\[ (i+n-2)\, P_{i+1}^{(n)}(t)=(2i+n-2)\, t\, P_i^{(n)}(t)-i\, P_{i-1}^{(n)}(t)
                \mbox{ for } i \geq 1. \]
We note that $\{P_i^{(n)}(t)\}$ are orthogonal in $[-1,1]$ with a weight $(1-t^2)^{(n-3)/2}$ and
satisfy $P_i^{(n)}(1)=1$ for all $i$ and $n$.
If $f(t) \in \mathbb{R}[t]$ is a real polynomial of degree $r$, then
$f(t)$ can be uniquely expanded in terms of the Gegenbauer
polynomials as
\begin{equation}\label{GegenbauerExpantion}
f(t) = \sum_{i=0}^r f_iP_i^{(n)}(t).
\end{equation}

We use the identity (see, for example, \cite[Corollary 3.8]{DGS}, \cite[Equation (1.7)]{Lev2}, \cite[Equation (1.20)]{Lev3})
\begin{equation}
  \label{main}
  |C|f(1)+\sum_{x,y\in C, x \neq y} f(\langle x,y\rangle)
      = |C|^2f_0 + \sum_{i=1}^r \frac{f_i}{r_i} \sum_{j=1}^{r_i}
        \left ( \sum_{x\in C} Y_{ij}(x) \right )^2
\end{equation}
as a source of estimations by polynomial techniques.
Here $C \subset \mathbb{S}^{n-1}$ is a spherical code, $f$ is as in \eqref{GegenbauerExpantion},
$\{ Y_{ij}(x) : j=1,2,\ldots,r_i\}$ is an orthonormal basis of the
space $\mathrm{Harm}(i)$ of homogeneous harmonic polynomials of
degree $i$ and $r_i=\dim \,\mathrm{Harm}(i)$.

The Delsarte-Goethals-Seidel bound  and the Levenshtein bound described in the next subsections are obtained after the sums on both sides of (\ref{main}) are neglected for suitable polynomials.

\subsection{Delsarte-Goethals-Seidel bound for spherical designs}

Denote by $D(n,\tau)$ the Delsarte-Goethals-Seidel \cite{DGS} bound for spherical designs
\begin{equation}
\label{DGS-bound}
B(n,\tau) \geq D(n,\tau) := \left\{ \begin{array}{l}
 \ds 2\binom{n+k-2}{n-1}, \mbox{ if $\tau=2k-1$,} \\[12pt]
 \ds \binom{n+k-1}{n-1}+\binom{n+k-2}{n-1}, \mbox{ if  $\tau=2k$},
\end{array}
  \right.
\end{equation}
where $B(n,\tau):=\min\{|C|: C \subset \mathbb{S}^{n-1} \mbox{ is a spherical $\tau$-design}\}$.
This bound plays an important role in our (initially heuristic)  choice of
 applications of Theorems~\ref{Thm3.1} and \ref{Thm3.6}.

We shall utilize the values of the function $D(n,\tau)$ to decide the degrees of the polynomials
to be used for the bounding of $\mathcal{L}(n,N,\tau;h)$ and $\mathcal{U}(n,N,\tau;h)$.
The rule is the following -- if we have dimension $n$, strength $\tau$ and cardinality
$N \in \left(D(n,\tau),D(n,\tau+1)\right]$, then we use polynomials
of degree $\tau$ for the lower bounds and $\tau$ or $\tau-1$ (depending on the parity of $\tau$)
for the upper bounds.

\subsection{Levenshtein bounds for spherical codes}

We now formulate and discuss the Levenshtein bounds \cite{Lev79,Lev2,Lev3,Lev} on
\begin{equation} \label{LevB}
A(n,s):=\max\{|C|: C \subset \mathbb{S}^{n-1}, \langle x,y \rangle \leq s \mbox{ for all } x,y \in C, x \neq y\},
\end{equation}
the maximal possible cardinality of a spherical code on $\mathbb{S}^{n-1}$ of prescribed maximal
inner product $s$.

For every positive integer $m$ we consider the intervals
\begin{eqnarray*}
  \mathcal{I}_m :=
\left\{
\begin{array}{ll}
    \left [ t_{k-1}^{1,1},t_k^{1,0} \right ], & \mbox{if } m=2k-1, \\[8pt]
    \left [ t_k^{1,0},t_k^{1,1} \right ],      & \mbox{if } m=2k. \\
  \end{array}\right.
\end{eqnarray*}
Here $t_0^{1,1}=-1$ by definition, $t_i^{a,b}$, $a,b \in \{0,1\}$, $i \geq 1$, is the largest zero of the Jacobi polynomial $P_i^{(a+\frac{n-3}{2},b+\frac{n-3}{2})}(t)$. The intervals $\{\mathcal{I}_m\}_{m=1}^\infty$ are well defined
(see \cite[Lemmas 5.29 and 5.30]{Lev}) and therefore constitute a partition of $\mathcal{I}=[-1,1)$ into
countably many closed subintervals with nonoverlapping interiors.

For every $s \in \mathcal{I}_m$, Levenshtein introduces a certain polynomial
$f_m^{(n,s)}(t)$ of degree $m$ that satisfies all the conditions of the corresponding
linear programming bounds for spherical codes and yields the bound
\begin{equation}
\label{L_bnd}
 A(n,s) \leq
\left\{
\begin{array}{ll}
    L_{2k-1}(n,s) := {k+n-3 \choose k-1}
         \big[ \frac{2k+n-3}{n-1} -
          \frac{P_{k-1}^{(n)}(s)-P_k^{(n)}(s)}{(1-s)P_k^{(n)}(s)}
         \big] \mbox{ for } s \in \mathcal{I}_{2k-1}, \\[10pt]
    L_{2k}(n,s) := {k+n-2 \choose k}
        \big[ \frac{2k+n-1}{n-1} -
           \frac{(1+s)( P_k^{(n)}(s)-P_{k+1}^{(n)}(s))}
    {(1-s)(P_k^{(n)}(s)+P_{k+1}^{(n)}(s))} \big] \mbox{ for } s \in \mathcal{I}_{2k}. \cr
    \end{array}\right.
\end{equation}
The function
\begin{equation}\label{Lns} L(n,s) :=
\left\{
\begin{array}{ll}
    L_{2k-1}(n,s) & \mbox{ for } s \in \mathcal{I}_{2k-1}, \\[6pt]
   L_{2k}(n,s)   & \mbox{ for } s \in \mathcal{I}_{2k} \cr
    \end{array}\right.
\end{equation}
is continuous in $s$.

The connections between the Delsarte-Gothals-Seidel bounds and the Levenshtein bounds
are given by the equalities
\begin{equation}
\label{L-DGS1}
L_{2k-2}(n,t_{k-1}^{1,1})=
L_{2k-1}(n,t_{k-1}^{1,1}) = D(n,2k-1) = 2{n+k-2 \choose n-1},
\end{equation}
\begin{equation}
\label{L-DGS2}
  L_{2k-1}(n,t_k^{1,0})=
L_{2k}(n,t_k^{1,0}) = D(n,2k) = {n+k-1 \choose n-1}+{n+k-2 \choose n-1}
\end{equation}
occurring at the ends of the intervals $\mathcal{I}_m$.

\subsection{A useful quadrature}

It follows from \cite[Section 4]{Lev} (see also \cite{Lev3,BBD}) that
for every fixed (cardinality) $N > D(n,2k-1)$ there exist
uniquely determined real numbers $-1 < \alpha_0 < \alpha_1 <
\cdots <\alpha_{k-1} < 1$ and $\rho_0,\rho_1,\ldots,\rho_{k-1}$,
$\rho_i>0$ for $i=0,1,\ldots,k-1$, such that the quadrature formula
\begin{equation}
\label{defin_f0.1}
f_0=\frac{\Gamma(n-1)}{2^{n-2}\Gamma(\frac{n-1}{2})^2}\int_{-1}^1 f(t)(1-t^2)^{\frac{n-3}{2}}\, dt= \frac{f(1)}{N}+ \sum_{i=0}^{k-1} \rho_i f(\alpha_i)
\end{equation}
holds for every real polynomial $f(t)$ of degree at most $2k-1$.

The number $\alpha_{k-1}=\alpha_{k-1}(N)$ in \eqref{defin_f0.1} is the solution of the equation $L(n,s)=N$ for $s>t_{k-1}^{1,1}$. Once it is determined, it is known (and easily verified) that the remaining $\alpha_i=\alpha_i (N)$, $i=0,1,\ldots,k-2$, are roots of the
equation $$P_k(t)P_{k-1}(\alpha_{k-1}) - P_k(\alpha_{k-1})P_{k-1}(t)=0,$$
where $P_i(t)=P_i^{(n-1)/2,(n-3)/2}(t)$ is a Jacobi polynomial of degree $i$.
In fact, $\alpha_i$, $i=0,1,\ldots,k-1$,
are the roots of the Levenshtein polynomial \cite[Eqs. (5.81) and (5.82)]{Lev}
(see also \cite[Theorem 5.39]{Lev})
\[ f_{2k-1}^{(n,\alpha_{k-1})}(t)=(t-\alpha_{k-1})\prod_{i=0}^{k-2} (t-\alpha_i)^2 \]
used for obtaining the bound $L_{2k-1}(n,s)$, $s=\alpha_{k-1}$, in (\ref{L_bnd}).

Similarly, for every fixed $N > D(n,2k)$ there exist
uniquely determined real numbers $-1=\beta_0 < \beta_1 < \cdots <\beta_k < 1$
and $\gamma_0,\gamma_1,\ldots,\gamma_k$, $\gamma_i>0$ for $i=0,1,\ldots,k$, such that the
quadrature formula
\begin{equation}
\label{defin_f0.2}
f_0= \frac{f(1)}{N}+ \sum_{i=0}^{k} \gamma_i
f(\beta_i)
\end{equation}
is true for every real polynomial $f(t)$ of degree at most $2k$.
The numbers $\beta_i=\beta_i(N)$, $i=0,1,\ldots,k$, are the roots of the Levenshtein
polynomial
\[ f_{2k}^{(n,\beta_{k})}(t)=(t-\beta_0)(t-\beta_k)\prod_{i=1}^{k-1} (t-\beta_i)^2 \]
(used for $L_{2k}(n,s)$, $s=\beta_k$, in (\ref{L_bnd})).

As mentioned in the Introduction we always take into consideration where the cardinality $N$ is located
with respect to the Delsarte-Goethals-Seidel bound.
It follows from the properties of the bounds $D(n,\tau)$ and $L_{\tau}(n,s)$
(see (\ref{L-DGS1}) and (\ref{L-DGS2})) that
\[ N \in [D(n,\tau),D(n,\tau+1)] \iff s \in {\mathcal I}_{\tau}, \]
where $n$, $s$ and $N$ are connected by the equality
\[ N=L_{\tau}(n,s). \]
Therefore we can always associate $N$ with the corresponding numbers:
\begin{equation}
\label{alphadef} \alpha_0,\alpha_1,\ldots,\alpha_{k-1},\rho_0,\rho_1,\ldots,\rho_{k-1} \mbox{ when } N \in \left(D(n,2k-1),D(n,2k)\right]
\end{equation}
or with
\begin{equation}
\label{betadef} \beta_0,\beta_1,\ldots,\beta_{k},\gamma_0,\gamma_1,\ldots,\gamma_{k} \mbox{ when } N \in \left(D(n,2k),D(n,2k+1)\right]. 
\end{equation}

\subsection{Bounds on smallest and largest inner products of spherical designs}

Denote
\begin{equation}\label{unNtau} u(n,N,\tau):=\sup \{u(C):C \subset \mathbb{S}^{n-1} \mbox{ is a $\tau$-design}, |C|=N\},
\end{equation}
where $u(C):=\max \{\langle x,y \rangle : x,y \in C, x \neq y\}$, and
\begin{equation}\label{lnNtau}\ell(n,N,\tau):=\inf \{\ell(C):C \subset \mathbb{S}^{n-1} \mbox{ is a $\tau$-design}, |C|=N\}, \end{equation}
where $\ell(C):=\min \{\langle x,y \rangle : x,y \in C, x \neq y\}$.

For every $n$, $\tau$ and cardinality $N \in [D(n,\tau),D(n,\tau+1)]$ non-trivial upper bounds on $u(n,N,\tau)$
can be obtained (cf. \cite{BBD,BBKS}). Similarly, for even $\tau=2k$ and cardinality $N \in [D(n,2k),D(n,2k+1))$
non-trivial lower bounds on $\ell(n,N,2k)$ are possible \cite{BBD,BBKS}. We describe here explicitly
the cases $\tau=2$ and $\tau=4$.

%

In \cite{BBD,BBKS} the quantities $u(n,N,\tau)$ and $\ell(n,N,\tau)$ are estimated
by using the following equivalent definition of spherical designs, which is a consequence of \eqref{main}: a spherical $\tau$-design $C \subset \mathbb{S}^{n-1}$ is a spherical code such that
for any point $x \in \mathbb{S}^{n-1}$ and any real
polynomial of the form  $f(t)=\sum_{i=0}^{\tau} f_iP_i^{(n)}(t)$, the equality
\begin{equation}
\label{defin_f.3} \sum_{x \in C} f(\langle x,y \rangle ) = f_0|C|
\end{equation}
holds.

\medskip

\begin{lemma}[\cite{BBKS}]
\label{Lem2.1} Let $n\ge 3$.

\begin{enumerate}
\item[{\rm (a)}] For  every $N \in [D(n,2),D(n,3)]=[n+1,2n]$ we have
\begin{equation}\label{ubnd2} u(n,N,2) \leq \frac{N-2}{n}-1.
\end{equation}

\item[{\rm (b)}] For every  $N \in [D(n,4),D(n,5)]=[n(n+3)/2,n(n+1)]$ we have
\begin{equation}\label{ubnd4} u(n,N,4) \leq \frac{2(3+\sqrt{(n-1)[(n+2)N-3(n+3)]})}{n(n+2)}-1.
\end{equation}
\end{enumerate}
\end{lemma}

%

\begin{proof}
To prove (a) one can use the polynomial $f(t)=\left(t+\sqrt{\frac{2}{n(N-2)}}\right)^2$ in (\ref{defin_f.3})
with $y$ being the midpoint of the geodesic arc between two of the closest points in $C$ as in
Theorem 3.2 from \cite{BBKS}. The bound in (b) is proved in the example following Theorem 3.2 from \cite{BBKS}.
\end{proof}

\medskip

\begin{lemma}[\cite{BBKS}]
\label{Lem2.2} Let $n\ge 3$.
\begin{enumerate}
\item[{\rm (a)}] For every  $N \in [D(n,2),D(n,3))=[n+1,2n)$ we have
\begin{equation}\label{lbnd2} \ell(n,N,2) \geq 1-\frac{N}{n}.
\end{equation}
\item[{\rm (b)}] For every  $N \in [D(n,4),D(n,5))=[n(n+3)/2,n(n+1))$ we have
\begin{equation}\label{lbnd4}\ell(n,N,4) \geq 1-\frac{2}{n}\left(1+\sqrt{\frac{(n-1)(N-2)}{n+2}}\right).
\end{equation}
\end{enumerate}
\end{lemma}

\begin{proof}
The bound in (a) can be proved by utilizing the polynomial $f(t)=t^2$ in (\ref{defin_f.3})
with $y$ being the midpoint of the geodesic arc between points $-x$ and $z \in C$, where $x$, $z$ is a
pair of  points of $C$ with smallest inner product (see Theorem 3.3 from \cite{BBKS}).
The bound in (b) is proved as in the example after Theorem 3.3 from \cite{BBKS}.
\end{proof}

\medskip

We remark that different bounds on $u(n,N,\tau)$ and $\ell(n,N,\tau)$ can be obtained by a technique
from \cite[Sections 2 and 3]{BBD}. For $\tau \geq 4$ such bounds are better in higher
dimensions than those from Lemmas \ref{Lem2.1} and  \ref{Lem2.2}.   More generally, when $\tau=2k$  we establish  in Lemma~\ref{Lem2.3}
lower bounds on $\ell(n,N,\tau)$ that will be used in Section 4 to establish Theorem~\ref{Thm4.1}.

\begin{lemma}[{\cite[Lemma 4.1]{BBD}}] \label{Lem2.3}
 Let $N \in (D(n,2k),D(n,2k+1))$ and $$f(t)=(t-\beta_1)^2\cdots(t-\beta_k)^2,$$ where $\beta_0,\ldots, \beta_k$ are as in \eqref{betadef}. 
 Let $\xi$ and $\eta$ denote the smallest and largest roots, respectively, of $f(t)=\gamma_0 Nf(-1)$. 
Then $\xi\le \ell(n,N,2k) \le u(n,N,2k) \leq \eta$.
\end{lemma}

\begin{proof}
For the convenience of the reader, we provide a proof of Lemma~\ref{Lem2.3}.  Let $C$ be a spherical design of strength $\tau=2k$ on ${\mathbb S}^{n-1}$ with $|C|=N$.  
Let $x$ and $y$ be distinct points in $C$.     Since  $f(t)\ge 0$ for all $t$  and $f$ vanishes at $\beta_1,\ldots, \beta_k$,  it follows from   \eqref{defin_f.3}    and \eqref{defin_f0.2}  that
\begin{equation*}
f(1)+f(\langle x,y\rangle)  \le \sum_{z\in C}f(\langle x,z\rangle)=Nf_0 = f(1)+N\gamma_0f(-1),
\end{equation*}
and so $f(\langle x,y\rangle)\le N\gamma_0 f(-1)$. 
Since $f(t)$ is strictly decreasing for $t<\beta_1$ and strictly increasing for $t>\beta_k$, we must have $\xi\le \langle x,y\rangle\le \eta$.  
\end{proof}

Note that   will produce a non-trivial bound  $\xi>-1$ if and only if $\gamma_0N<1$  and the next assertion,   implicit in \cite[Section 4]{BDL} (see
also Remark 5.58 in \cite{Lev}),  shows that this is indeed true for $N \in (D(n,2k),D(n,2k+1))$.    

\medskip

\begin{lemma}
\label{Lem2.4}
If $N \in (D(n,2k),D(n,2k+1)),$ then $\gamma_0 N \in (0,1)$.  Hence, $\ell(n,N,2k)>-1$.  
\end{lemma}

\begin{proof}
We have the formulas (Equation (5.113) in \cite{Lev})
\[ \gamma_0=\frac{T_k(s,1)}{T_k(-1,-1)T_k(s,1)-T_k(-1,1)T_k(s,-1)} \]
and (the equation in the last line of page 488 in \cite{Lev})
\[ N=L_{2k}(n,s)=\frac{T_k(1,1)T_k(s,-1)-T_k(1,-1)T_k(s,1)}{T_k(s,-1)}. \]
A little algebra then shows that
\[ \gamma_0 N= \frac{T-A(s)}{T-1/A(s)}, \]
where $A(s)=T_k(s,1)/T_k(s,-1)$ as in \cite{BDL} and $T=T_k(1,1)/T_k(1,-1)$.
Moreover, we have
\[ A(s)=T \cdot \frac{P_k^{1,0}(s)}{P_k^{0,1}(s)} \]
from \cite[Lemma 5.24]{Lev}, where $P_k^{1,0}(s)>0$ and $P_k^{0,1}(s)<0$ for every
$s \in \left(t_k^{1,0},t_k^{1,1}\right)$ (see Lemmas 5.29 and 5.30 in \cite{Lev}).
Therefore the signs of $A(s)$ and $T$ are opposite. We conclude that
\[  \gamma_0 N= \frac{|T|+|A(s)|}{|T|+1/|A(s)|}. \]
The ratio $\frac{P_k^{1,0}(s)}{P_k^{0,1}(s)}$ is
decreasing in $s$ in the interval $\left(t_k^{1,0},t_k^{1,1}\right)$ (see \cite[Lemma 5.31]{Lev}), i.e.
$|A(s)|$ is increasing in $s \in \left(t_k^{1,0},t_k^{1,1}\right)$.
Since $\gamma_0 N=0$ and 1 for $s=t_k^{1,0}$ and $t_k^{1,1}$, respectively,
we obtain that $\gamma_0 N$ increases from 0 to 1 when $s$ increases from $s=t_k^{1,0}$ to $t_k^{1,1}$.
\end{proof}

\medskip

\section{General lower and upper bounds}

The general framework of the linear programming bounds for the quantities $\mathcal{L}(n,N,\tau;h)$ and $\mathcal{U}(n,N,\tau;h)$ is
given by the next two theorems.  Theorem~\ref{Thm3.1} is an adaptation of   known results (cf. \cite{Y,CK}) to spherical designs, we are not aware
of any prior use of linear programming techniques for obtaining upper bounds on
the potential energy as in Theorem~\ref{Thm3.6}.

\medskip

\begin{theorem} \label{Thm3.1}
 Let $n$, $N$, $\tau$ be positive integers with $N \geq D(n,\tau)$ and let $h:[-1,1]\to[0,+\infty]$.     Suppose
 $I$ is a subset of $[-1,1)$ and $f(t) = \sum_{i=0}^{\deg(f)} f_i P_i^{(n)}(t)$ is a real
polynomial such that
\begin{enumerate}
\item[{\rm (A1)}] $f(t) \leq h(t)$ for $t\in I$, and

\item[{\rm (A2)}] the Gegenbauer coefficients  satisfy $f_i \geq 0$ for $i \geq \tau+1$.
\end{enumerate}
If $C \subset \mathbb{S}^{n-1}$ is a spherical $\tau$-design of $|C|=N$ points such that $\langle x,y\rangle \in I$ for distinct points
$x,y\in C$, then
\begin{equation}\label{mainEbnd}
E(n,C;h) \geq N(f_0N-f(1)).
\end{equation}
In particular, if $I=[\ell(n,N,\tau),u(n,N,\tau)]$, then
\begin{equation}\label{mainLbnd}
\mathcal{L}(n,N,\tau;h) \geq N(f_0N-f(1)).
\end{equation}
\end{theorem}
\medskip

\begin{proof} Using (\ref{main}) and the conditions of the theorem we consecutively have
\begin{eqnarray*}
Nf(1)+E(n,C;h) &=& Nf(1)+\sum_{x,y\in C, x \neq y} h(\langle x,y \rangle) \\
               &\geq& |C|f(1)+ \sum_{x,y\in C, x \neq y} f(\langle x,y\rangle) \\
               &=& |C|^2f_0 + \sum_{i=1}^{\deg(f)} \frac{f_i}{r_i} \sum_{j=1}^{r_i} \left ( \sum_{x\in C} Y_{ij}(x) \right )^2 \\
               &\geq& N^2f_0,
\end{eqnarray*}
which implies that $E(n,C;h) \geq N(f_0N-f(1))$.
If $I=[\ell(n,N,\tau),u(n,N,\tau)]$, the assertion \eqref{mainLbnd} follows immediately from the definitions  \eqref{unNtau} and  \eqref{lnNtau}.
\end{proof}
  \medskip

Our choice of polynomials for Theorem~\ref{Thm3.1} follows from ideas of Levenshtein \cite{Lev3,Lev} and
the connections (\ref{L-DGS1}) and (\ref{L-DGS2}).
We start with fixed dimension $n$, cardinality $N$ and strength $\tau$ under the assumption that
$N \in \left(D(n,\tau),D(n,\tau+1)\right]$. Now the equation $N=L_{\tau}(n,s)$ determines all necessary
parameters as explained in subsections 2.3 and 2.4.

\begin{defn} We denote by $A_{n,\tau,I;h}$ the set of
polynomials satisfying conditions (A1) and (A2) of Theorem~\ref{Thm3.1}.  For convenience we shall write
$A_{n,\tau;h}:=A_{n,\tau,[-1,1];h}$.
\end{defn}

Next we {\em need} Hermite interpolation as follows. If $h\in C^1([-1,1])$, we define
the Hermite interpolant $F(t)$ as follows:

(i) for odd $\tau=2k-1$ the polynomial $F(t)$ of degree at most $2k-1$ by
\[ F(\alpha_i)=h(\alpha_i), \ F^\prime(\alpha_i)=h^\prime(\alpha_i), \ i=0,1,\ldots,k-1; \]

(ii) for even $\tau=2k$ the polynomial $F(t)$ of degree at most $2k$ by
\[ F(\beta_0)=h(\beta_0), \ F(\beta_i)=h(\beta_i), \ F^\prime(\beta_i)=h^\prime(\beta_i), \ i=1,\ldots,k. \]

These conditions define, as in \cite{CK} (see also \cite{And,KY,Y}), a
Hermite interpolation problem  that requires the graph of $F(t)$ to intersect and be tangent to the
graph of the potential function $h(t)$ at all points $\alpha_i$ and all $\beta_i$ except for $\beta_0=-1$ where only intersection is required.

\medskip

%
%
%
%
%
\begin{lemma}\label{Lem4.1}
  Suppose $h$ is absolutely monotone on $[-1,1]$ and that $F$ satisfies {\rm (i)} or {\rm (ii)}.
  Then  $F(t)\le h(t)$ for all $t\in[-1,1]$.
\end{lemma}
\begin{proof}
The proof follows from the well-known error formula for Hermite interpolation (see \cite{D}), namely
\begin{equation}
\label{HermiteError}
h(t)-F(t) = \begin{cases} \frac{h^{(\tau+1)}(\xi )}{(\tau+1)!}(t-\alpha_0)^2\cdots (t-\alpha_{k-1})^2, & \tau=2k-1,\\
& \\
  \frac{h^{(\tau+1)}(\xi )}{(\tau+1)!}(t-\beta_0)(t-\beta_1)^2\cdots (t-\beta_{k})^2, & \tau=2k,
\end{cases}
 \end{equation}
for some $\xi=\xi(t)\in(-1,1)$ and the fact that $h^{(\tau+1)}(t)\ge 0$ for $t\in [-1,1]$.
\end{proof}

\medskip

In \cite[Theorem~3.1]{BDHSS1}, the following result was established by the authors, 
which provides a lower bound on the $h$-energy of {\it general} codes.

\begin{theorem} \label{Thm3.4} Let $n,N$ and $\tau$ be positive integers with 
$N \in (D(n,\tau),D(n,\tau+1)]$  and  suppose   $h$ is absolutely monotone on $[-1,1]$.   
Then, for any positive integer $\tau$, we have
\begin{equation}
\label{bound_odd}
\mathcal{L}(n,N,\tau;h)\ge \inf_{C} E(n,C;h) \geq \begin{cases}N^2\sum_{i=0}^{k-1} \rho_i h(\alpha_i),& \tau=2k-1, \\
& \\
 N^2\sum_{i=0}^{k} \gamma_i h(\beta_i),& \tau=2k,
\end{cases}
 \end{equation}
 where the infimum is over spherical codes $C\subset \Sp^{n-1}$   with $|C|=N$.
  \end{theorem}

\medskip

Utilizing Theorems \ref{Thm3.1} and \ref{Thm3.4} we provide a sufficient condition 
for the optimality of the bounds in \eqref{bound_odd} for a class of spherical designs.
\begin{theorem} \label{Thm3.5}
Suppose  $n$, $\tau$, and $I$ are as in Theorem~\ref{Thm3.1}, $N\in (D(n,\tau),D(n,\tau+1)]$, $h$  is absolutely monotone on $[-1,1]$ and that
$ \alpha_i\in I$ for $i=0,\ldots , k-1$ in the case that $\tau =2k-1$ and that $ \beta_i\in I$ for $i=0,\ldots , k$ in the case that $\tau =2k$.
If $C\subset \Sp^{n-1}$ is a spherical $\tau$-design with $|C|=N$ and inner products $\langle x, y \rangle\in I$ for $x\neq y\in C$, then the linear programming lower bounds in \eqref{bound_odd} cannot be improved by utilizing polynomials  of degree at most  $\tau$
satisfying {\rm (A1)}; i.e., for any such polynomial  $f$  we have
\begin{equation}\label{Ibnd}
N(f_0N-f(1))\le N(F_0N-F(1))=\begin{cases}N^2\sum_{i=0}^{k-1} \rho_i h(\alpha_i),& \tau=2k-1, \\
& \\
N^2\sum_{i=0}^{k} \gamma_i h(\beta_i),& \tau=2k,
\end{cases}
\end{equation}
where $F(t)$ is the Hermite interpolating polynomial from {\rm (i)} and {\rm (ii)}.
\end{theorem}


\medskip

\begin{proof} We shall consider only the case $\tau=2k-1$ since the $\tau=2k$ case is analogous. 
Notice that {\rm (i)} and \eqref{defin_f0.1} allow us to rewrite \eqref{bound_odd}  as
\[ E(n,C;h) \geq N^2\sum_{i=0}^{k-1} \rho_i F(\alpha_i)=N(F_0 N-F(1)).\]
Lemma~\ref{Lem4.1} implies that $F(t) \leq h(t)$ for every $t \in [-1,1]$; in particular $F(t)$
satisfies the condition (A1) of Theorem~\ref{Thm3.1}.
The condition (A2) is trivially satisfied and therefore $F \in A_{n,2k-1,I;h}$.


Furthermore, for any polynomial $f(t) \in A_{n,2k-1,I;h}$ of degree at most $2k-1$,
we have from the quadrature formula (\ref{defin_f0.1}) for $f(t)$ and the fact that $\{\alpha_i\}\subset I$
\[ N(F_0N-F(1))=N^2\sum_{i=0}^{k-1} \rho_i h(\alpha_i) \geq N^2\sum_{i=0}^{k-1} \rho_i f(\alpha_i)=N(f_0N-f(1)),\]
which proves \eqref{Ibnd} and the theorem.
\end{proof}

\medskip

In Theorem~\ref{Thm4.1}  we show that  the  bound for $\mathcal{L}(n,N,\tau;h)$ in \eqref{bound_odd}
 can be improved over  the whole range
$D(n,\tau) < N < D(n,\tau+1)$ in the case of even $\tau$.
\medskip

As mentioned above, the specific properties of the spherical designs, namely the existence of nontrivial
upper bounds on $u(n,N,\tau)$, allows further application
of the linear programming techniques. We are able to derive upper bounds on $\mathcal{U}(n,N,\tau;h)$
thus setting a strip for the energies of the spherical designs under consideration.

\medskip

\begin{theorem}\label{Thm3.6}  Let $n$, $N$, $\tau$ be positive integers with $N \geq D(n,\tau)$ and let $h:[-1,1]\to[0,+\infty]$.     Suppose
 $I$ is a subset of $[-1,1)$ and $g(t) = \sum_{i=0}^{\deg(g)} g_i P_i^{(n)}(t)$ is a real polynomial such that

\begin{enumerate}
\item[{\rm (B1)}] $g(t) \geq h(t)$ for  $ t \in I$, and

\item[{\rm (B2)}] the Gegenbauer coefficients   satisfy $g_i \leq 0$ for $i \geq \tau+1$.
\end{enumerate}
If $C \subset \mathbb{S}^{n-1}$ is a spherical $\tau$-design of $|C|=N$ points such that $\langle x,y\rangle \in I$ for distinct points
$x,y\in C$, then
\begin{equation}\label{mainEUbnd}
E(n,C;h) \leq N(g_0N-g(1)).
\end{equation}
In particular, if $I=[\ell(n,N,\tau),u(n,N,\tau)]$, then
\begin{equation}\label{mainUbnd}
\mathcal{U}(n,N,\tau;h) \leq N(g_0N-g(1)).
\end{equation}
\end{theorem}

\begin{proof} Let $C \subset \mathbb{S}^{n-1}$ be an arbitrary spherical
$\tau$-design of $|C|=N$ points. Using (\ref{main}) and the conditions of the theorem we consecutively have
\begin{eqnarray*}
Ng(1)+E(n,C;h) &=& Ng(1)+\sum_{x,y\in C, x \neq y} h(\langle x,y \rangle) \\
               &\leq& Ng(1)+\sum_{x,y\in C, x \neq y} g(\langle x,y\rangle) \\
               &=& N^2g_0 + \sum_{i=1}^{\deg(g)} \frac{g_i}{r_i} \sum_{j=1}^{r_i} \left( \sum_{x\in C} Y_{ij}(x) \right )^2 \\
               &\leq& N^2g_0,
\end{eqnarray*}
which implies that $E(n,C;h) \leq N(g_0N-g(1))$. Since the design $C$
was arbitrary, we conclude that $\mathcal{U}(n,N,\tau;h) \leq N(g_0N-g(1))$.
\end{proof}
  \medskip

We utilize Theorem \ref{Thm3.6} to determine an upper bound that in conjunction 
with the lower bound \eqref{Ibnd} in Theorem \ref{Thm3.5} defines a strip where 
the energy of designs lives.  We shall formulate the theorem for the odd case $\tau=2k-1$, 
but a similar assertion holds for the even case $\tau=2k$. 
\begin{theorem}\label{Thm3.7}  
Let $n$, $N$, $\tau=2k-1$ be positive integers with $N \in (D(n,\tau),D(n,\tau+1)]$ 
and let $h:[-1,1]\to[0,+\infty]$ be  absolutely monotone.  For  
 $\alpha_{k-1}<u<1$ (see \eqref{alphadef}) and every $j \in \{0,1\ldots,k-1\}$,
let $G(t)=G_{j,u}(t)$ to be the Hermite interpolant of $h$ of degree $2k-1$ that satisfies
\[ G(\alpha_i)=h(\alpha_i), \ G^\prime(\alpha_i)=h^\prime(\alpha_i), \ i \in \{0,1,\ldots,k-1\} \setminus \{j\},
G(-1)=h(-1), \ G(u)=h(u). \]

Then, for any spherical $\tau$-design $C$ with $|C|=N$ and $\displaystyle u(C)=\max_{x,y\in C, x\neq y}\langle x,y\rangle \leq u$,  
\begin{equation}\label{StripUbnd}
E(n,C;h) \leq N(G_0N-G(1))=N^2\sum_{i=0}^{k-1} \rho_i h(\alpha_i)+N^2\rho_{j}\left[G(\alpha_{j})-h(\alpha_{j})\right].
\end{equation}
In particular, taking $u:=u(n,N,\tau)$, we obtain 
\begin{equation}\label{StripEstimate}
\mathcal{U}(n,N,\tau;h) -  \mathcal{L}(n,N,\tau;h) \leq N^2 \min_{0\le j\le k-1} \rho_{j}[G_{j,u}(\alpha_{j})-h(\alpha_{j})].
\end{equation}
\end{theorem}

\begin{proof} 
It follows from the Hermite error formula \[ h(t)-G(t)=\frac{h^{(2k)}(\xi )}{(2k)!}(t+1)(t-u)\prod_{i\not= j}(t-\alpha_i)^2, \]
that $G(t)\geq h(t)$ for $t\in[-1,u]$. 
We next apply the quadrature formula \eqref{defin_f0.2} to the polynomial $G(t)$ and utilize 
Theorem \ref{Thm3.6} to conclude \eqref{StripUbnd}. The estimate  \eqref{StripEstimate} follows from Theorem \ref{Thm3.4} and \eqref{StripUbnd}.
\end{proof}


\section{Improving the linear programming  lower bounds}

\subsection{Even strength}

We show that the bounds from Theorem~\ref{Thm3.4} can be improved when some additional information about the
distribution of the inner products of the designs under consideration is available.
This is exactly the case for designs of even strength $2k$ and cardinality $N$ in the
interval $(D(n,2k),D(n,2k+1))$. In fact, $N=D(n,2k)$ is possible only for $k=1$ (the regular simplex) and $k=2$ (see \cite{BD1,BD2}). \medskip

It follows from Lemmas~\ref{Lem2.3} and \ref{Lem2.4} that $-1< \ell(n,N,2k)$ for every $N \in (D(n,2k),D(n,2k+1))$.
We next use this fact to improve the bound (\ref{bound_odd})
for the case of even strength $\tau=2k$, where $N \in (D(n,2k),D(n,2k+1))$.

\begin{theorem}
\label{Thm4.1} Suppose $N \in (D(n,2k),D(n,2k+1))$ and $h$ is absolutely monotone on $[-1,1]$.  Let $G(t)$ be the Hermite interpolant of $h(t)$ of degree at most $2k$ such that 
\[ G(\ell)=h(\ell), \ G(\beta_i)=h(\beta_i), \ G^\prime(\beta_i)=h^\prime(\beta_i), \ i=1,\ldots,k, \]
where $\ell:=\ell(n,N,2k)$.
Then
\[ \mathcal{L}(n,N,2k;h) \geq N(G_0N-G(1))=N\sum_{i=0}^k \gamma_i G(\beta_i)>N^2\sum_{i=0}^k \gamma_i h(\beta_i). \]
\end{theorem}

\medskip

\begin{proof}
Using Theorem \ref{Thm3.1} with $f=G$ and $I=[\ell,1]$, we obtain $\mathcal{L}(n,N,2k;h) \geq N(G_0N-G(1))$.
Since the degree of $G(t)$ is at most  $2k$  we can apply the quadrature formula \eqref{defin_f0.2} and so
\begin{eqnarray*}
G_0N-G(1) &=& N\sum_{i=0}^k \gamma_i G(\beta_i)=N\left(\gamma_0G(-1)+\sum_{i=1}^k \gamma_i h(\beta_i)\right) \\
&=& N\left(\gamma_0(G(-1)-h(-1))+\sum_{i=0}^k \gamma_i h(\beta_i)\right) \\
&>& N \sum_{i=0}^k \gamma_i h(\beta_i),
\end{eqnarray*}
(the inequality $G(-1)>h(-1)$ follows from the interpolation). \end{proof}

\medskip

We remark that the   choice of the polynomial $G(t)$ in Theorem~\ref{Thm4.1} is not usually optimal for maximizing the lower bound in Theorem~\ref{Thm4.1}.  In a forthcoming paper, we shall develop methods for choosing optimal interpolation points.  

Next we show that the inequality for $\ell(n,N,2)$ from Lemma~\ref{Lem2.2}(a) can be used for obtaining a lower bound for
$\mathcal{L}(n,N,2;h)$ that is  better than (\ref{bound_odd}) in the whole range $n+1=D(n,2) < N < 2n=D(n,3)$.
\medskip

\begin{theorem}\label{Thm4.2}   Let $n\ge 2$ and $N \in [n+1,2n]$.  If $h$ is absolutely monotone on $[-1,1]$, then
\begin{equation}
\label{2lower}
\mathcal{L}(n,N,2;h) \geq \frac{N[h(0)N(N-n-1)+nh(1-N/n)]}{N-n}.
\end{equation}
In particular, if $\zeta:=N/n$, then
\begin{equation}
\label{2lowerb}
\mathcal{L}(N/\zeta,N,2;h) \geq h(0)N^2+\frac{N[h(1-\zeta)-\zeta h(0)]}{\zeta-1}.
\end{equation}
\end{theorem}

\medskip

\begin{proof} Let $\kappa\le \ell(n,N,2)$. We apply Theorem~\ref{Thm3.1} with $I=[\kappa,1]$ and  a second degree polynomial $f(t)$ such that $f(\kappa)=h(\kappa)$, $f(a)=h(a)$ and
$f^\prime(a)=h^\prime(a)$ for some $a \in (\kappa,1)$. The optimal choice  of $a$  so as to maximize $f_0N-f(1)$ is
\begin{equation}
\label{a_opt_lb_deg3}
a_0:=\frac{n(1-\kappa)-N}{n(1-\kappa)+\kappa Nn}.
\end{equation}
By Lemma~\ref{Lem2.2}(a), $\ell(n,N,2)\ge 1-N/n$. Choosing $\kappa=1-N/n$ gives $ a_0=0$ and, hence,  the  bound \eqref{2lower} follows.
\end{proof}

\medskip

Lower bounds for the energy of 4-designs by Theorem~\ref{Thm3.1} can be obtained by interpolation with polynomials
of degree four:
$$f(\kappa)=h(\kappa), \ \ f(a)=h(a), \ \ f^\prime(a)=h^\prime(a), \ \ f(b)=h(b), \ \ f^\prime(b)=h^\prime(b),$$
where $\kappa$ is the expression on the right-hand side of \eqref{lbnd4} and the touching points $a$ and $b$ are chosen to maximize $f_0N-f(1)$.

\subsection{Odd strength}

In contrast to the even strength case, non-trivial lower bounds on $\ell(n,N,2k-1)$ seem impossible for $N\in  [D(n,2k-1),D(n,2k)]$
without further constraints.  Instead we show that 
higher degree polynomials can lead   to improving the lower bound (\ref{bound_odd}) for $\tau=2k-1$.

Let $n$, $N\in [D(n,2k-1),D(n,2k)]$ and $\tau=2k-1$ be fixed and $j$ be a positive integer.
Taking the necessary parameters from $N=L_{2k-1}(n,s)$ we consider the following
functions in $n$ and $s \in {\mathcal I}_{2k-1}$
\begin{equation}
\label{test-functions}
Q_j(n,s):= \frac{1}{N}+\sum_{i=0}^{k-1} \rho_i P_j^{(n)}(\alpha_i).
\end{equation}
Note that $Q_j(n,s)$  is the same as the   quadrature rule expression on the right-hand side of \eqref{defin_f0.1} applied to  $f=P_j^{(n)}$ since $P_j^{(n)}(1)=1$.

The functions $Q_j(n,s)$ were firstly introduced and investigated in \cite{BDB}. The applications in
\cite{BDB} target the upper bounds on $A(n,s)$ and, in particular, the possibilities for improving the
Levenshtein bounds. We are going
to see that the functions  $Q_j(n,s)$ are useful for our purposes as well.
The next theorem shows that (the signs of) the functions $Q_j(n,s)$ give necessary and sufficient conditions for existence
of improving polynomials of higher degrees.

\medskip

\begin{theorem}\label{Thm4.3} Assume that $h$ is strictly absolutely monotone. Then the bound
(\ref{bound_odd}) can be improved by a polynomial from $A_{n,2k-1;h}$ of degree at least $2k$
if and only if one has $Q_j(n,s) < 0$ for some $j \geq 2k$.

Moreover, if $Q_j(n,s)<0$ for some $j \geq 2k$, then (\ref{bound_odd}) can be improved by a polynomial from $A_{n,2k-1;h}$
of degree exactly $j$.
\end{theorem}

\medskip

\begin{proof}  The necessity mirrors the spherical codes' analog
\cite[Theorem 4.1]{BDHSS1} (see also \cite[Theorem 2.6]{BDHSS1}).
We include the simplified proof of the sufficiency in the context of the spherical designs for completeness.

Let us assume $Q_j(n,s) <0$ for some $j \geq 2k$. We prove that the bound (\ref{bound_odd}) can
be improved by using the polynomial $ f(t)=\varepsilon P_j^{(n)}(t)+g(t)$
for suitable $\varepsilon>0$, where the polynomial $g(t)$ is such that $\deg(g)=2k-1$ and
\begin{equation}
\label{inter_H}
g(\alpha_i)=h(\alpha_i)-\varepsilon P_j^{(n)}(\alpha_i), \
g^\prime(\alpha_i)=h^\prime(\alpha_i)-\varepsilon(P_j^{(n)})^\prime(\alpha_i), \ i=0,1,\ldots,k-1
\end{equation}
(i.e. $g(t)$ is Hermite interpolant of $h(t)-\varepsilon P_j^{(n)}(t)$ in the points $\alpha_0,\alpha_1,\ldots,\alpha_{k-1}$.
We denote $g(t)=\sum_{\ell=0}^{2k-1} g_\ell P_\ell^{(n)}(t)$. Note that $f_0=g_0$ and $f(1)=g(1)+\varepsilon$.

We first prove that $f(t)=\varepsilon P_j^{(n)}(t)+g(t) \in A_{n,2k-1;h}$ for some $\varepsilon>0$.
For condition (A2) we need to see only that $f_j=\varepsilon>0$.
For condition (A1), let us choose $\varepsilon >0$ such that
$\left( h-\varepsilon P_j^{(n)}\right)^{(\ell)}(t) \geq 0$ for every $\ell \geq 0$ and for
every $t \in [-1,1]$. It is clear that such $\varepsilon$ exists because $h$ is strictly absolutely monotone and
this leaves finitely many $\ell$ to generate inequalities for $\varepsilon$.
Moreover, since the function $\tilde{h}(t):= h(t)-\varepsilon P_j^{(n)}(t)$ is absolutely monotone
by the choice of $\varepsilon$, we could infer as in Lemma~\ref{Lem4.1} that $g$ satisfies $g(t) \leq \tilde{h}(t)$
for every $t \in [-1,1)$ which implies that $f(t) \leq h(t)$ for every $t \in [-1,1]$ and hence
$f(t) \in A_{n,2k-1;h}$.

It remains to see that the bound given by $f(t)$ is better than (\ref{bound_odd}). This follows
by combining equalities from (\ref{inter_H}) and applying (\ref{defin_f0.1}) and (\ref{test-functions})
to obtain
\[ N(Nf_0-f(1))=N^2\sum_{i=0}^{k-1} \rho_i h(\alpha_i)-\varepsilon N^2 Q_j(n,s) \]
Since $Q_j(n,s)<0$, we have $N(Nf_0-f(1))>N^2\sum_{i=0}^{k-1} \rho_i h(\alpha_i)$, i.e.
the polynomial $f(t)$ indeed gives a better bound.
\end{proof}

\medskip

As mentioned above, the test functions $Q_j(n,s)$ were initially defined in \cite{BDB} as related to the Levenshtein bounds $\mathcal{L}(n,s)$
on maximal spherical codes. Theorem~\ref{Thm4.3} shows that $Q_j(n,s)$ prove to be very useful tool
in the context of potential energy as well.
In particular, the signs of the test functions $Q_{2k+3}(n,s)$, where $\tau=2k-1$, were investigated in detail in \cite{BDB}
(see also \cite{Bou01}). Denote
\[ k_0 := \frac{k^2-4k+5 + \sqrt{k^4-8k^3-6k^2+24k+25}}{4} \]
for short ($k_0$ is well defined for $k \geq 9$).

\medskip

\begin{theorem} \label{Thm4.4} \cite{Bou01} We have $Q_{2k+3}(n,s)<0$ for every $s \in
   \left(t_{k-1}^{1,1},t_k^{1,0}\right)$ and for every
$n \geq 3$ and $k \geq 9$ which satisfy $3 \leq n \leq k_0$.
\end{theorem}
\medskip

\begin{corollary} \label{Cor4.5} The bound (\ref{bound_odd}) can be improved by using polynomials of degree $2k+3$
for every $s \in \left(t_{k-1}^{1,1},t_k^{1,0}\right)$ and for every
$n \geq 3$ and $k \geq 9$ that satisfy $3 \leq n \leq k_0$.
\end{corollary}
\medskip

\begin{proof} This follows from Theorems~\ref{Thm4.3} and \ref{Thm4.4}.
\end{proof}

\medskip

\medskip

\begin{remark}\label{Rem5.7} We note that test functions
for the even case $\tau=2k$ can be defined and investigated as well (see \cite{BDB}). However, we
conjecture that the linear programming in $[\ell(n,N,\tau),u(n,N,\tau)]$ by Theorem~\ref{Thm3.1}
is always better than the bounds which would come from higher degree polynomials.
\end{remark}

\section{Improved upper bounds for   2, 3, and 4-designs}

In this section, we use bounds from Lemmas~\ref{Lem2.1} and \ref{Lem2.2} to specify an interval $I$  in Theorem~\ref{Thm3.6} to obtain upper bounds for the energy of  2-, 3- and 4-designs.  Our numerical experiments suggest that the use of even degree polynomials is not effective and so we turn
our attention to polynomials of  degrees 1 and 3.

\subsection{Upper bounds for 2-designs}

First we apply Theorem~\ref{Thm3.6} for $g$ a linear polynomial.  
\medskip

\begin{theorem} \label{Thm5.1}  Let $h$ be a convex non-negative function on $[-1,1]$ and let $u$ and $\ell$ denote the upper and lower bounds in \eqref{ubnd2} and \eqref{lbnd2}, respectively.   For $N \in [n+1,2n)$, if $\ell<u$ then
\begin{equation}
\label{2upper}
\mathcal{U}(n,N,2;h) \leq  \frac{N\big[(N-1)(uh(\ell)-\ell h(u))+h(\ell)-h(u)\big]}{u-\ell},
\end{equation}
if $\ell=u$ then  $\mathcal{U}(n,N,2;h)=Nh(-1/(N-1))$.
\end{theorem}

\medskip

\begin{proof}
With $I=[\ell,u]$, the linear polynomial   passing 
through the points $(\ell,h(\ell))$ and $(u,h(u))$
satisfies the conditions of Theorem~\ref{Thm3.6} and gives the desired bound.  If $u=\ell$, then we must also have $u=\ell=-1/(N-1)$ which implies that $\mathcal{U}(n,N,2;h)=Nh(-1/(N-1))$.
\end{proof}

\medskip

\begin{remark}
Combining  \eqref{2lower} and \eqref{2upper} gives a strip for the $h$-energy
of any spherical 2-design of $N \in \{n+1,n+2,\ldots,2n-1\}$ points when $h$ is absolutely monotone.
 Note that if $n$ and $N$ tend simultaneously to infinity such that $n/N\to \zeta$ for some $\zeta \in (1,2)$,
Theorem~\ref{Thm4.2} and Theorem~\ref{Thm5.1} give an asymptotic strip
  as $N\to \infty$:
\begin{equation}
\label{2strip_asymp}
\begin{split} h(0)+&\frac{h(1-\zeta)-\zeta h(0)}{(1-\zeta)N} +O\left( N^{-2}\right)\le \frac{\mathcal{U}(n,N,2;h)}{N^2}\\
&\le \frac{h(1-\zeta)+h(\zeta-1)}{2}+\frac{(2-\zeta)h(1-\zeta)-\zeta h(\zeta-1)}{2(\zeta-1) N} +O\left(N^{-2}\right).
\end{split}\end{equation}
\end{remark}

\medskip

\newpage

\begin{example} \label{Ex5.3} Simple algebraic manipulations show that the bounds
(\ref{2lower}) and (\ref{2upper}) for 2-designs coincide when $N=n+1$ or $N=n+2$ for every $n$ and $h$
(i.e. the strip becomes a point for these two cardinalities).
The case $N=n+1$ leads to the regular simplex on $\mathbb{S}^{n-1}$. 

The case $N=n+2$ is more interesting --
Mimura \cite{Mim} has proved that spherical 2-designs with $n+2$ points on $\mathbb{S}^{n-1}$
do exist if and only if $n$ is even and Sali
\cite{Sal92} (see also \cite{Sei69}) proved that there are no other (up to isometry) such 2-designs. 
Spherical $2$-designs of $N=2k$ points on $\mathbb{S}^{2k-3}$ are known as {\em Mimura spherical designs} and consist of two orthogonal $k$-simplices which we denote by $\{ k,k \}$.   This design has $k(k-1)$ distances of $\sqrt{2k/(k-1)}$ coming from edges within the two $k$-simplices and $k^2$  distances of $\sqrt{2}$ which are the edges joining the vertices from distinct simplices. The total number of various distances is $k(2k-1)={N \choose 2}$. 

Sali's  nonexistence result follows easily from the coincidence of our bounds.  
It also follows that the 2-designs of $n+2$ points for even $n$ are unique and optimal --
they have simultaneously minimum and maximum possible energy. The optimality cannot be extended to
the larger class of spherical codes -- Cohn and Kumar \cite[Proposition 1.4]{CK} prove that
if $n+1<N<2n$, then there is no $N$-point universally optimal spherical codes on $\mathbb{S}^{n-1}$.
For the cases, $N \in \{n+3,n+4,\ldots,2n-1\}$, there is a difference (increasing with $N$)
between the bounds from (\ref{2lower}) and (\ref{2upper}).

\end{example}\label{Ex5.4}
\begin{example}[Riesz $s$-energy of Mimura designs]
As mentioned above, Mimura designs clearly give the minimum energy over 2-designs with $2k$ points in $\mathbb{S}^{2k-3}$.
 The Riesz $s$-energy, that is the energy when $h(t)=(2(1-t))^{-s/2}$, is given by
\begin{equation}\label{MimuraEnergy} E_s (\{ k,k\})=\frac{k^2}{2^{s/2}}\left(1+\frac{k(k-1)}{k^2}\left( \frac{k-1}{k}\right)^{s/2}\right).\end{equation}

It is easy to see that Mimura designs are not universally optimal among general codes of cardinality $N=2k$.  
Indeed,  consider the competing configuration $\{2,2k-2 \}$  made of two orthogonal simplices, a $2$-simplex and a $(2k-2)$-simplex (say the North Pole-South Pole diameter  and a $(2k-2)$-simplex in the equatorial hyperplane).
This configuration has $1$ distance of length $2$, $(2k-2)(2k-3)/2$ distances of length $\sqrt{2(2k-2)/(2k-3)}$ and $2(2k-2)$ distances of length $\sqrt{2}$. Thus the $s$-energy of this competing configuration is  given by
\begin{equation} \label{CompetingEnergy} E_s (\{ 2,2k-2\})=\frac{2(2k-2)}{2^{s/2}}\left(1+\frac{1}{2(2k-2)2^{s/2}}+\frac{2k-3}{4}\left( \frac{2k-3}{2k-2}\right)^{s/2}\right),\end{equation}thus showing that for large enough $s$ the Mimura design $\{ k,k\}$ will have strictly larger energy than the competing configuration $\{2,2k-2\}$.
However, for $k=3$ it was shown in \cite{Dr} that the Mimura configuration minimizes, in particular, the logarithmic energy. 
\end{example}

\medskip

\subsection{Upper bounds for 3 and 4-designs}
For our estimates we shall apply 
  Theorem~\ref{Thm3.6} with $g(t)$ the Hermite interpolating polynomial of degree at most three satisfying
\begin{equation}\label{Hint3} g(\ell)=h(\ell), \ g(a)=h(a), \ g^\prime(a)=h^\prime(a), \ g(u)=h(u), 
\end{equation}
where $\ell$ and $u$  again denote lower and upper bounds, respectively, for the inner products
of all spherical 3- or 4-designs under consideration. For $\tau=3$ one can take $\ell=-1$ and $u$
as in \cite[Theorem 3.9]{BBD}. For $\tau=4$ the bounds $u$ and $\ell$ are taken from Lemma~\ref{Lem2.1}(b)
and \ref{Lem2.2}(b), respectively.

\medskip

\begin{theorem} \label{Thm6.5} Let $h\in C^4([-1,1])$ with $h^{(4)}(t)\ge 0$ for  $t\in[-1,1]$.
For $\tau=3$ and $N \in [2n,\frac{n(n+3)}{2})$, and for $\tau=4$  and $N \in [\frac{n(n+3)}{2},n^2+n)$, we have  \begin{eqnarray}
\label{ub_deg3}
\hspace*{4mm} \mathcal{U}(n,N,\tau;h) &\leq& N(N-1)h(a_0) \\
&& +\,\frac{(h(\ell)-h(a_0))\left[uN(1+na_0^2)+2Na_0+n(1-u)(1-a_0)^2\right]}{n(u-\ell)(\ell-a_0)^2} \nonumber \\
&& +\,\frac{(h(u)-h(a_0))\left[\ell N(1+na_0^2)+2Na_0+n(1-\ell)(1-a_0)^2\right]}{n(u-\ell)(u-a_0)^2}, \nonumber
\end{eqnarray}
  where
\begin{equation}
\label{a_opt_ub_deg3}
a_0:=\frac{N(\ell+u)+n(1-\ell)(1-u)}{n(1-\ell)(1-u)-N(1+\ell un)},
\end{equation}
and $u$ and $\ell$ are chosen as described above.  
\end{theorem}

\medskip

\begin{proof} As in Lemma~\ref{Lem4.1} we see that $g(t) \geq h(t)$ for every $t \in [\ell,u]$, i.e. the condition (B1) is satisfied.
Moreover, (B2) is trivially satisfied and therefore Theorem~\ref{Thm3.6} can be applied.
Standard calculations for optimization of $g_0N-g(1)$ via derivatives show
that the optimal value of $a$ is given by (\ref{a_opt_ub_deg3}).
\end{proof}
\medskip

The asymptotic form of the bound (\ref{ub_deg3}) for $\tau=4$ is easily determined when $n$ and $N \in [D(n,4),D(n,5))=[\frac{n(n+3)}{2},n^2+n)$ tend simultaneously to infinity as described in the following corollary.

\medskip

\begin{corollary}\label{Thm6.6}  If $n$ and $N$ tend to infinity in relation $N=n^2 \lambda+o(1)$ as $N\to \infty$,
where $\lambda \in [1/2,1)$ is a constant, then
\begin{equation}
\label{asymp4des}
\mathcal{U}(n,N,4;h) \le h(0)N^2-h(0)N+c_1\sqrt{N}+c_2+o(1), \qquad (N\to \infty),
\end{equation}
where
$$ c_1=\frac{\sqrt{\lambda}[(2\sqrt{\lambda}-1)h(1-2\sqrt{\lambda})+(1-2\sqrt{\lambda})h(2\sqrt{\lambda}-1)]}{2(2\sqrt{\lambda}-1)^3}$$
and
$$ c_2=\frac{(1-\sqrt{\lambda})h(1-2\sqrt{\lambda})+\sqrt{\lambda}h(2\sqrt{\lambda}-1)-h(0)}{(2\sqrt{\lambda}-1)^3}.$$
\end{corollary}

\medskip

\begin{proof} The assertion follows from  Theorem~\ref{Thm6.5} and the asymptotic formulas below:
\begin{eqnarray*}
u(n,N,4) &=& 2\sqrt{\lambda}-1+o(1) \ \ \mbox{(from Lemma~\ref{Lem2.1}(b))}, \\
\ell(n,N,4) &=& 1-2\sqrt{\lambda} +o(1)\ \ \mbox{(from Lemma~\ref{Lem2.2}(b))}, \\
a_0 &=& o(1) \ \ \mbox{(from (\ref{a_opt_ub_deg3}))}.
\end{eqnarray*}
\end{proof}

\medskip


\section{Some asymptotic lower bounds}

We consider the bounds (\ref{bound_odd})
in the asymptotic process where the strength $\tau$ is fixed,
and the dimension $n$ and the cardinality $N$ tend simultaneously to infinity in the relation
\begin{equation}
\label{asy-1}
\lim_{n,N \to \infty} \frac{N}{n^{k-1}}=\frac{2}{(k-1)!}+\gamma
\end{equation}
(here $\gamma \geq 0$ is a constant and the term $\frac{2}{(k-1)!}$ comes from the Delsarte-Goethals-Seidel bound).

\medskip

\begin{theorem}\label{Thm7.1} Let $h$ be absolutely monotone on $[-1,1]$ and $\tau$ be fixed.   If $n$ and $N$ tend to infinity as in {\rm  \eqref{asy-1}}, then  
\begin{equation}
\label{l-bound-a1}
 \mathcal{L}(n,N,\tau;h)\ge  h(0) N^2+o(N^2).
\end{equation}
\end{theorem}

\begin{proof} Let $\tau=2k-1$. From \eqref{asy-1} and 
  \cite{BoumDan}, the asymptotic behavior of the
parameters from (\ref{bound_odd}) is as follows: $\alpha_i =o(1),$ for $i = 1, 2, \ldots, k-1$,
$\alpha_0 = -\frac{1}{1+\gamma (k-1)!}+o(1)$, $\rho_0N = (1+\gamma (k-1)!)^{2k-1}+o(1)$ as $N\to \infty$.
Now the lower bound of (\ref{bound_odd}) is easily  calculated:
\begin{eqnarray*}
\mathcal{L}(n,N,2k-1;h) &\geq& N^2 \sum_{i=0}^{k-1} \rho_i h(\alpha_i) \\
&=& N^2\left(\rho_0h(\alpha_0)+h(0)\sum_{i=1}^{k-1} \rho_i\right) +o(N^2)\\
&=& N^2\left(\rho_0(h(\alpha_0)-h(0))+h(0)\left(1-\frac{1}{N}\right)\right)+o(N^2) \\
&=& h(0)N^2+o(N^2).
\end{eqnarray*}
 
Similarly, in the even case $\tau=2k$, we obtain using the appropriate asymptotic relations from \cite{BoumDan},
\begin{eqnarray*}
\mathcal{L}(n,N,2k;h) &\ge& N^2\left(\gamma_0(h(-1)-h(0))+h(0)\left(1-\frac{1}{N}\right)\right) +o(N^2)\\
&=& h(0)N^2+o(N^2) 
\end{eqnarray*}
\end{proof}

\medskip

Note that for $\tau=4$ the main terms in the lower bound (\ref{l-bound-a1}) and the
upper bound (\ref{asymp4des}) coincide.  For odd $\tau$, (\ref{l-bound-a1}) also yields good energy estimates as shown in the next example.


\medskip

\begin{example} \label{Ex7.2} There is standard construction (see \cite[Chapter 5]{CS}) mapping binary codes from the Hamming space $H(n,2)$ to the
sphere $\mathbb{S}^{n-1}$ -- the coordinates $0$ and $1$ are replaced by $\pm 1/\sqrt{n}$, respectively.
Denote by $\overline{x}$ and $\overline{C}$ the images of vector $x \in H(n,2)$ and code $C \subset H(n,2)$.
Then the inner product $\langle \overline{x},\overline{y} \rangle$ on $\mathbb{S}^{n-1}$
and the Hamming distance $d_H(x,y)$ in $H(n,2)$ are connected by
\[ \langle \overline{x},\overline{y} \rangle = 1-\frac{2d_H(x,y)}{n}. \]

Levenshtein \cite[pages 67-68]{Lev3} shows that spherical codes which are obtained by this construction from the
Kerdock codes \cite{Ker72} are asymptotically optimal with respect to their cardinality. So it is natural to
check them for energy optimality.

The Kerdock codes $K_{\ell} \subset H(2^{2\ell},2)$ are nonlinear.
They exist in dimensions $n=2^{2\ell}$ and their cardinality is $N=n^2=2^{4\ell}$. The Hamming distance (weight) distribution
does not depend on the point and is as follows:
\begin{eqnarray*}
A_i &=& 0 \ \mbox{for $i \neq 0, i_1=2^{2\ell-1}-2^{\ell-1}, i_2=2^{2\ell-1}, i_3=2^{2\ell-1}+2^{\ell-1}, n=2^{2\ell}$}, \\
A_0 &=& 1, \\
A_{i_1} &=& 2^{2\ell}(2^{2\ell-1}-1), \\
A_{i_2} &=& 2^{2\ell+1}-2, \\
A_{i_3} &=& 2^{2\ell}(2^{2\ell-1}-1), \\
A_n &=& 1.
\end{eqnarray*}
The weights $0,i_1,i_2,i_3,n$ correspond to the inner products $1,\frac{1}{\sqrt{n}},0,-\frac{1}{\sqrt{n}},-1$, respectively.
The spherical code $\overline{K}_{\ell} \subset \mathbb{S}^{2^{2\ell}-1}$ has energy
\[ E(n,\overline{K}_{\ell};h) = N \left(\left(2^{2\ell+1}-2\right)h(0)+
2^{2\ell}(2^{2\ell-1}-1)\left(h\left(\frac{1}{\sqrt{n}}\right)+h\left(-\frac{1}{\sqrt{n}}\right)\right)+h(-1)\right). \]
When $n$ tends to infinity we obtain
\[  E(n,\overline{K}_{\ell};h) = \frac{N^2}{2}\left( h(\frac{1}{\sqrt{n}})+h(-\frac{1}{\sqrt{n}})\right)+O(N) =h(0)N^2+O(N^{3/2})=h(0)n^4+O(n^3), \]
where we assumed that $h$ is differentiable at 0.  \end{example}


\begin{thebibliography}{99}

\bibitem{And96} N.\,N.\,Andreev, An extremal property of the icosahedron, \emph{East J. Approx.} 2, 459-462, 1996.

\bibitem{And} N.\,N.\,Andreev, Location of points on a sphere with minimal energy, \emph{Tr. Math. Inst. Steklova}
{\bf 219} 27-31, 1997 (in Russian); English translation: {\em Proc. Inst. Math Steklov} {\bf 219}, 20-24, 1997.

\bibitem{And99} N.\,N.\,Andreev, A spherical code, \emph{Uspekhi Mat. Nauk} 54, 255-256, 1999 (in Russian);
English translation: {\em Russian Math. Surveys} {\bf 54}, 251-253, 1999.

\bibitem{BD1} E.\,Bannai, R.\,M.\,Damerell, Tight spherical designs I,
{\it J. Math. Soc. Japan} {\bf 31}, 1979, 199-207.

\bibitem{BD2} E.\,Bannai, R.\,M.\,Damerell, Tight spherical designs II,
{\it J. London Math. Soc.} {\bf 21}, 1980, 13-30.

\bibitem{BBMQ} B.\,Beckermann, J.\,Bustamante, R. Martinez-Cruz, J. Quesada,
Gaussian, Lobatto and Radau positive quadrature rules with a prescribed abscissa, {\em Calcolo} {\bf 51}, 319--328, (2014).

%
%

\bibitem{BRV1} A.\,Bondarenko, D.\,Radchenko, M. Viazovska,  Optimal asymptotic bounds for spherical
designs, \emph{Ann. Math.} {\bf 178}, 2013, 443-452.

\bibitem{BRV2} A.\,Bondarenko, D.\,Radchenko, M. Viazovska,  Well-separated spherical designs,
\emph{Constr. Approx.} {\bf 15}, 2015, 93--112.



\bibitem{BHSbook}
S.\,Borodachov, D.\,Hardin, E.\,Saff,
\emph{Minimal Discrete Energy on Rectifiable Sets},
manuscript.

\bibitem{Bou01} S.\,Boumova, Applications of polynomials to spherical codes
and designs, PhD Dissert., TU Eindhoven, 2002.

\bibitem{BBKS} S.\,Boumova, P.\,Boyvalenkov, H.\,Kulina, M.\,Stoyanova,
Polynomial techniques for investigation of spherical designs,
\emph{Des., Codes Crypt.}, {\bf 51} 275-288, 2009.

\bibitem{BoumDan} S.\,Boumova, D.\,Danev, On the asymptotic behaviour of
a necessary condition for existence of spherical designs, Proc.
Intern. Workshop ACCT, Sept. 8-14, 2002, 54-57.

\bibitem{BBD} P.\,Boyvalenkov, S.\,Boumova, D.\,Danev, Necessary conditions
for existence of some designs in polynomial metric spaces, \emph{Europ. J. Combin.}, {\bf 20} 213-225, 1999.

\bibitem{BDB}
P.\,Boyvalenkov, D.\,Danev, S.\,Bumova, \emph{Upper bounds on the
minimum distance of spherical codes}, IEEE Trans. Inform. Theory \textbf{42},
1996, 1576-1581.

\bibitem{BDL} P.\,Boyvalenkov, D.\,Danev, I.\,Landjev, On
maximal spherical codes II, \emph{J. Combin. Des.} 7, 1999, 316-326.

\bibitem{BDHSS1} P.\,Boyvalenkov, P.\,Dragnev, D.\,Hardin, E.\,Saff, M.\,Stoyanova,
Universal energy bounds for potential energy of spherical codes,  (submitted).

\bibitem{CK} H.\,Cohn, A.\,Kumar,
Universally optimal distribution of points on spheres,
\emph{J. of Amer. Math. Soc.}, {\bf 20} no. 1, 99-148, 2007.

\bibitem{CS} J.\,H.\,Conway, N.\,J.\,A.\,Sloane, {\it Sphere Packings, Lattices
and Groups}, Springer-Verlag, New York, 1988.

\bibitem{D} P.\,J.\,Davis, \emph{Interpolation and Approximation}, Blaisdell Publishing Company, New York,
1963.

\bibitem{D1} P.\,Delsarte, {\it An Algebraic Approach to the
Association Schemes in Coding Theory}, Philips Res. Rep. Suppl. 10, 1973.

\bibitem{DGS}
P.\,Delsarte, J.-M.\,Goethals, J.\,J.\,Seidel, Spherical codes and designs,
\emph{Geom. Dedicata,} {\bf 6} 363-388, 1977.

\bibitem{Dr} P. D. Dragnev,  Log optimal configurations on the sphere, \emph{Contemporary Math}, AMS.


\bibitem{FL} G.\,Fazekas, V.\,Levenshtein, On the upper bounds for
code distance and covering radius of designs in polynomial metric
spaces, \emph{J. Comb. Theory A} {\bf 70}, 1995, 267-288.

\bibitem{HS} D.\,P.\,Hardin, E.\,B.\,Saff, Discretizing manifolds via minimum energy points,
\emph{Notices Amer. Math. Soc.}, {\bf 51}, no. 10, 1186-1194, 2004.

\bibitem{H} K.\,Hesse, The $s$-energy of spherical designs on $\mathbb{S}^2$,
\emph{Adv. Comput. Math.}, {\bf 30}, 37-59, 2009.

\bibitem{HL} K.\,Hesse, P.\,Leopardi, The Coulomb energy of spherical designs on $\mathbb{S}^2$,
\emph{Adv. Comput. Math.}, {\bf 28}, 331-354, 2008.

\bibitem{KL} G. A. Kabatyanskii, V. I. Levenshtein, Bounds for packings
on a sphere and in space, {\it Probl. Inform. Transm.} {\bf 14}
(1989), 1-17.

\bibitem{Ker72} A.\,M.\,Kerdock, A class of low-rate nonlinear binary codes,
{\it Inform. and Control} 20, 1972, 182-187.

\bibitem{KY94} A.\,V.\,Kolushov, V.\,A.\,Yudin, On the Korkin-Zolotarev construction,
\emph{Diskret. Mat.} 6, 155-157, 1994 (in Russian) English translation in \emph{Discrete
Math. Appl.} {\bf 4}, 143--146, 1994.

\bibitem{KY} A.\,V.\,Kolushov, V.\,A.\,Yudin, Extremal dispositions of points on the sphere,
\emph{Anal. Math.}, {\bf 23}, 25-34, 1997.

\bibitem{Lev79} V.\,I.\,Levenshtein, On bounds for packings in n-dimensional Euclidean space,
\emph{Dokl. Akad. Nauk SSSR} 245, 1299-1303, 1979 (in Russian); English translation in
\emph{Soviet Math. Dokl.} {\bf 20}, 417-421, 1979.

\bibitem{Lev2} V.\,I.\,Levenshtein, Bounds for packings in metric spaces and
certain applications, {\it Probl. Kibernetiki} {\bf 40}, 1983, 44-110 (in Russian).

\bibitem{Lev3} V.\,I.\,Levenshtein, Designs as maximum codes in polynomial
metric spaces, {\it Acta Appl. Math.} {\bf 25}, 1992, 1-82.


\bibitem{Lev}
V.\,I.\,Levenshtein,
Universal bounds for codes and designs, \emph{Handbook of Coding Theory},
V.\,S.~Pless and W.\,C.~Huffman, Eds., Elsevier, Amsterdam, 1998, Ch.~6, 499--648.

\bibitem{MW} W.\,J.\,Matrin, J.\,S\,Williford, There are finitely many $Q$-polynomial association schemes
with given first multiplicity at least three, \emph{Europ. J. Combin.} {\bf 30}, 2009, 698-704.

\bibitem{Mim} Y.\,Mimura, A construction of spherical 2-designs, \emph{Graphs Combin.}, {\bf  6}, 1990, 369-372.

\bibitem{Sal92} A.\,Sali, On the rigidity of some spherical 2-designs, \emph{Mem. Fac. Sci. Kyushu Univ.}
Ser. A {\bf 47} (1), 1993, 1-14.

\bibitem{Sei69} J.\,J.\,Seidel, Quasiregular two-distance sets, \emph{Nederl. Akad. Wetensch. Proc.}
Ser. A {\bf 72}, 1969, 64-70 (\emph{Indag. Math.} {\bf 31}).

\bibitem{SK} E.\,B.\,Saff, A.\,B.\,J.\,Kuijlaars,
Distributing many points on a sphere, \emph{Math. Intelligencer,} {\bf 19}, 1997, 5-11.

\bibitem{Sze}
G.\,Szeg\H{o}, \emph{Orthogonal Polynomials},
Amer. Math. Soc. Col. Publ., {\bf 23}, Providence, RI, 1939.

\bibitem{Y} V.\,A.\,Yudin, Minimal potential energy of a point system of charges,
\emph{Discret. Mat.} {\bf 4}, 115-121, 1992 (in Russian); English translation: {\em Discr. Math. Appl.} {\bf 3}, 75-81, 1993.
\end{thebibliography}
\end{document}